\documentclass{amsart}

\usepackage[utf8]{inputenc}
\usepackage{amsfonts}
\usepackage{hyperref}
\usepackage{amsmath}
\usepackage{amsfonts}
\usepackage{amssymb}
\usepackage{xcolor}
\usepackage{amsthm}
\usepackage{pdflscape}
\usepackage{pgfplots}
\pgfplotsset{compat=1.18}
\usepackage{mathrsfs}
\usepackage{array}
\usepackage{cases}
\usepackage{graphicx}
\usepackage{chngcntr}
\usepackage{tikz-cd} 
\usepackage{caption}
\usepackage{subcaption}
\usepackage{color,soul}

\newtheorem{thm}{Theorem}[section]

\newtheorem{prop}[thm]{Proposition}
\newtheorem{lem}[thm]{Lemma}

\newtheorem{quest}[thm]{Question}

\theoremstyle{definition}
\newtheorem{defn}[thm]{Definition}

\newtheorem{exmp}[thm]{Example}

\theoremstyle{remark}
\newtheorem{rem}[thm]{Remark}

\makeatletter
\let\c@equation\c@thm
\makeatother
\numberwithin{equation}{section}

\DeclareMathOperator{\Sing}{Sing}
\DeclareMathOperator{\coreg}{coreg}

\renewcommand{\P}[1]{\ensuremath{\mathbb{P}^{#1}}}

\newcommand{\sfrac}[2]{\left(\frac{#1}{#2}\right)}

\newcommand{\pfrac}[2]{\frac{\partial#1}{\partial#2}}

\renewcommand{\phi}{\varphi}
\renewcommand{\leq}{\leqslant}
\renewcommand{\geq}{\geqslant}

\calclayout

\counterwithin{figure}{section}

\definecolor{utahcrimson}{rgb}{0.83, 0.0, 0.25}

\definecolor{emerald}{rgb}{0.31, 0.78, 0.47}

\hypersetup{
    colorlinks=true,
    citecolor=emerald,
    linkcolor=utahcrimson,
    filecolor=magenta,      
    urlcolor=cyan,
    pdftitle={Overleaf Example},
    pdfpagemode=FullScreen,
    }
    
\title{On smooth Fano threefolds with coregulatiry zero}

\author{Olzhas Zhakupov}

\address{Moscow Institute of Physics and Technology,
9 Institutskiy per., Dolgoprudny, Moscow Region, 141701, Russia}

\email{zhakupov.ob@phystech.edu}

\begin{document}
\begin{abstract}
We provide examples of smooth three-dimensional  Fano complete intersections of dergee $2, 4, 6$, and $8$ that have coregularity $0$. Considering the main theorem of \cite{ALP23} on the remaining $101$ families of smooth Fano threefolds, our result implies that each family of smooth Fano threefolds has an element of coregularity zero.  
\end{abstract}
\maketitle
\tableofcontents
\section{Introduction}

 All varieties in this text are considered over $\mathbb{C}$. 
 
 A projective variety $X$ is called \textit{Fano} if its anticanonical bundle $\omega_X^{\vee}$ is ample. Fano varieties play an important role in the birational classification of algebraic varieties. They appear as the fibres of Mori fibre spaces, which are the end result of a minimal model program for a variety covered by rational curves.
  Koll\'ar, Miyaoka and Mori proved in \cite{KMM92} that smooth Fano varieties form finitely many deformation families in every dimension.
  For example, any Fano curve is isomorphic to $\P{1}$. The smooth Fano surfaces are called del Pezzo surfaces. They form $10$ families and are classified according to the degree of the anti-canonical divisor  $(-K_X)^2$ which is an integer in the set $\{1,\ldots,9\}$.  Iskovskikh, Mori and Mukai obtained the classification of smooth Fano threefolds in $105$ deformation families, see \cite{IP99}. Finding smooth elements in the anti-canonical linear system $\left|-K_X\right|$ played a fundamental role in the latter classification. In general, in order to study Fano varieties, it is helpful to understand how singular such elements could be.  The notion of \textit{coregularity} $\coreg(X)$, inspired by \cite{Sh00} and developed in \cite{Mo22} allows to ``measure" this concept on a Fano variety $X$. The coregularity is defined in terms of Calabi-Yau pairs (see Definition \ref{defn:2.1}). Given a Fano variety X, one can get a Calabi-Yau pair, by taking an \textit{$l$-complement} of $K_X$, that is, a boundary $\mathbb{Q}$-divisor $D$ such that $l(K_x + D) \sim 0$. The \emph{dual complex} $\mathcal{D}(X,D)$ (see Definition \ref{defn:2.3}) of a Calabi-Yau pair is a topological space that reflects the combinatorics of the boundary divisor $D$ as well as the geometry of $X$.
 We say that a Calabi-Yau pair $(X,D)$ has maximal intersection, if its dual complex has dimension equal to $\dim X - 1$. In terms of coregularity, this means that $\coreg(X)=0$. Such Fano varieties are far from the exceptional ones, which have $\coreg(X) = \dim X$. Being exceptional here means that any lc log Calabi-Yau structure is in fact klt. 
 
 It is expected that ``most'' Fano varieties have coregularity zero. For example, $\coreg (\P{1}) = 0$. Further, any smooth del Pezzo surface of degree at least ~$2$ has coregularity zero (\cite[Theorem 3.2]{Mo22}) and a general del Pezzo surface of degree $1$ also has coregularity zero (\cite[Proposition 2.3]{ALP23}). The case of the smooth Fano threefolds was thoroughly studied in \cite{ALP23} (see Theorem \ref{thm:2.5}). It turned out that for $100$ out of $105$ families of smooth Fano threefolds, a general element has coregularity zero. For $92$ out of these $100$ families, any element has coregularity zero. The remaining 5 families are $1.1, 1.2, 1.3, 1.4$ and 1.5 in the notation of  \cite{IP99}. 
 In \cite[Lemma 6.1]{ALP23} it is shown that there exist a smooth Fano threefold from the family $1.5$ of coregularity zero. The family 1.2 is the family of smooth quartics $X \subset \P{4}$. 
 A.-S. Kaloghiros in \cite{Ka18} provides the example of a smooth quartic of coregularity zero:
\begin{thm}[{\cite{Ka18}}] \label{thm:1.1}
    There exist a smooth quartic $X \subseteq \P{4}$ of coregularity zero.     
\end{thm}
Our main result is as follows:
\begin{thm} \label{thm:1.2}
There exist a smooth sextic double solid $X \subseteq \mathbb{P}(1,1,1,1,3)$, a smooth intersection of a quadric and a cubic $X_{2, 3} \subseteq \P{5}$ and a smooth intersection of three quadrics $X_{2,2,2} \subseteq \P{6}$ of coregularity zero. 
\end{thm}
Smooth Fano threefolds from the Theorem \ref{thm:1.2} belong to the families 1.1, 1.3, 1.4, respectively. Combining Theorem \ref{thm:1.1} and Theorem \ref{thm:1.2} with the main result of \cite{ALP23} -- Theorem \ref{thm:2.5}, we get:
\begin{thm} \label{thm:1.3}
In any family of smooth Fano threefolds, there exists an element of coregularity zero.
\end{thm}
As a possible continuation of studying the coregularity of smooth Fano varieties, it is interesting to understand whether the behavior of Theorem \ref{thm:1.3} extends to higher dimensions:
\begin{quest}
    Does every deformation family of smooth $n$-dimensional Fano varieties contain an element of coregularity zero?
\end{quest}

\textbf{Acknowledgements. }I am grateful to my advisor Konstantin Loginov for suggesting this problem as well as for his patience and invaluable support. 
\section{Preliminaries}

In what follows, we use the standard notions of birational geometry, see \cite{KM98}. We remind that all varieties are considered over $\mathbb{C}$.

\subsection{Pairs and singularities}

A \textit{pair} $(X, D)$ consists of a normal variety $X$ and a $\mathbb{Q}$-Weil divisor $D$ with coefficients in $[0,1]$ such that $K_X + D$ is $\mathbb{Q}$-Cartier. The divisor $D$  is called a \textit{boundary} divisor. If $f\colon Y \to X$ is a birational morphism, then a \textit{log pullback} $(Y, D_Y)$ of a pair $(X, D)$ is defined by the following formulas:
\[
K_Y + D_Y = f^{*}(K_X + D), \quad \quad \quad \quad f_*(D_Y)=D.
\]
In this case we will also say that $D_Y$ a log pullback of $D$. Note that generally a log pullback $(Y, D_Y)$ is not a pair since the coefficients of $D_Y$ may not belong to $[0,1]$.  
For a divisor $E \subset Y$ its \textit{log discrepancy}  $a(E, X, \Delta)$ is defined as 
\[a(E, X, D) = 1 - \mathrm{coeff}_{E}D_Y.\] 
A pair $(X, D)$ is called \textit{log canonical} or, shortly, \textit{lc} if $a(E, X, D) \geq 0$ for every divisor $E$ and every birational morphism $ f \colon Y \to X$. A pair $(X, D)$ is called \textit{Kawamata log terminal} or, shortly, \textit{klt} if $a(E, X, D) > 0$ for every divisor $E$ and every birational morphism $ f \colon Y \to X$. A pair $(X, D)$ is called \textit{divisorially log terminal} or, shortly, dlt if $a(E, X, D) > 0$ for every $f$-exceptional divisor $E$ and for some log-resolution $ f \colon Y \to X$. 
%\subsection{Calabi-Yau pairs}
A birational morphism $f \colon Y \to X$ is called a \textit{log resolution} of the pair $(X,D)$ if $Y$ is smooth and $\mathrm{Exc}(f)\cup \mathrm{supp}(f^{-1}_*D)$ is a simple normal crossing divisor. Two pairs $(X_1, D_1),\, (X_2, D_2)$ are called \textit{crepant equivalent} if there are proper birational morphisms $f_{i}\colon Y \to X_{i}$, $i=1,2$ such that the log pullback of $D_1$ is equal to the log pullback of $D_2$. 

We use the following definition of a Calabi-Yau pair:

\begin{defn} \label{defn:2.1}
    An lc pair $(X, D)$ is called a \textit{Calabi-Yau pair} if $K_{X} + D \sim_{\mathbb{Q}} 0$.    
\end{defn} 
\begin{exmp} \label{exmp:2.2}
    Let $X = \P{2}$ and $D = L + C$ be the union of a line and a conic intersecting transversally. Then such pair is lc and $K_X + D \sim_{\mathbb{Q}} 0$. 
\end{exmp}
\subsection{Dual complex and coregularity}

Suppose $X$ is a normal variety and a divisor $D = \sum_{i\in I} D_i$ on it has a simple normal crossing support. The \textit{dual complex} $\mathcal{D}(D)$ is  a regular CW-complex defined as follows. The irreducible components $Z$ of the intersection $\bigcap_{i \in J} D_i,\,J \subseteq ~I$, $|J| = k + 1$ are in bijection with $k$-simplices $v_Z$. For $j \in J$, $Z$ is contained in some unique irreducible component $W$ of the intersection $\bigcap_{i \in J \setminus \{j\}} D_i$. This inclusion gives a gluing map $v_W \to v_Z$ which is an inclusion of the face that doesn't contain the vertex $v_j$. The dimension of a dual complex is its dimension as a CW-complex. 

We denote the  sum of components which have a coefficient $1$ in a divisor $D$ as $D^{=1}$. 
\begin{defn} \label{defn:2.3}
    A \textit{dual complex} $\mathcal{D}(X,D)$ of a Calabi-Yau pair $(X,D)$ is defined as $\mathcal{D}(X,D) = \mathcal{D}(D_Y^{=1})$, where $f\colon (Y,D_Y) \to (X,D)$ is a log resolution. 
\end{defn}
\begin{rem}
    It is proved in \cite{dFKX17} that the PL homeomorphism class of the dual complex $\mathcal{D}(X,D)$ does not depend on the choice of a log resolution and, more generally, does not change under a crepant equivalence.    
\end{rem} Note that $\dim \mathcal{D}(X,D)$ is not greater than  $\dim X - 1$, since an snc divisor has intersections of at most $\dim X$ components at every point.  
\begin{exmp}
    Suppose $(X,D)$ is a Calabi-Yau pair from the Example \ref{exmp:2.2}. Since $X$ is smooth and $D = L + C$ is snc, the dual complex $\mathcal{D}(X, D) \cong S^1$ (Fig. \ref{fig:2.1}).
\end{exmp}
\begin{figure}[h!]
\centering
\begin{tikzpicture}[scale = 0.5]
\draw[black] (0,0) circle (2);
\filldraw[black] (2,0) circle (2.5pt);
\filldraw[black] (-2,0) circle (2.5pt);
\draw node at (-2.5, -0.5){$v_L$};
\draw node at (2.5, -0.5){$v_C$};
\draw node at (0, 2.5){$v_1$};
\draw node at (0, -2.5){$v_2$};
\end{tikzpicture}
\caption{Dual complex of a pair $(\P{2},L+C)$}
\label{fig:2.1}
\end{figure}
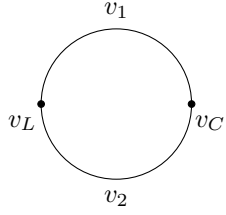
The notion of \textit{regularity} $\mathrm{reg}(X, D)$ of an lc pair was introduced in \cite{Sh00}:
\[\mathrm{reg}(X,D) = \dim \mathcal{D}(X, D).\]
In the case $\mathcal{D}(X, D) = \varnothing$ we say by convention that $\dim \varnothing = -1$.
The notion of regularity of a normal projective variety $X$ was defined in \cite{Mo22}. By definition, it is equal to
\begin{equation*}
    \label{reg}\mathrm{reg}(X) = \mathrm{max}\,\{\dim\mathcal{D}(X,D)\,|\,(X,D) \text{ is a Calabi-Yau pair} \}.
\end{equation*}
The dual notion of \textit{coregularity} is defined by the following formula:
\[\mathrm{coreg}(X) = \dim X - \mathrm{reg}(X) - 1.
\]
The main result of \cite{ALP23} about the coregularity of the smooth Fano threefolds is as follows:
\begin{thm}[{\cite[Theorem 0.1]{ALP23}}] \label{thm:2.5}
    Let $X$ be a smooth Fano threefold. Put
    \begin{align*}
    &\aleph = \{1.1,\,1.2\}, \qquad \beth = \{1.3,\,1.4\},\\
    &\gimel = \{1.5\}, \qquad\qquad \daleth = \{1.6,\, 1.7,\,1.8,\,1.9,\,1.10,\,1.11,\,2.1,\,10.1\}.
    \end{align*}
    Then the following holds.
    \begin{enumerate}
        \item  If $X$ is any smooth threefold that belongs to any family except for the families in $\aleph,\, \beth,\,\gimel,\,\daleth,$ then $\coreg(X) = 0$.
        \item  If $X$ is a general element in one of the families $\aleph$, we have $\coreg(X) \geq 1$.
        \item If $X$ is a general element in one of the families $\beth$, we have $\coreg_1(X) = 2$.
        \item If $X$ is a general element in  the family $\gimel$, we have $\coreg(X) \leq 1$.
        \item If $X$ is a general element in one of the families $\daleth$, we have $\coreg(X) = 0$.
    \end{enumerate}
\end{thm}
\begin{rem}
    The notion $\mathrm{coreg}_1(X)$ in Theorem \ref{thm:2.5} is the \textit{first coregularity} of a Fano variety $X$ (see \cite[Definition 1.2]{ALP23}). It arises when one only considers the boundary divisors $D$ from the anti-canonical linear system $\left|-K_X\right|$ in the definition of the regularity: 
    \begin{gather*}
        \mathrm{reg}_1(X) = \mathrm{max}\,\{\dim\mathcal{D}(X,D)\,|\,D \in \left|-K_X\right|\},\\
        \mathrm{coreg}_1(X) = \dim X - \mathrm{reg}_1(X) - 1.
    \end{gather*} 
\end{rem}
Note that to prove that a smooth Fano threefold $X$ has $\coreg X = 0$, it is sufficient to find a boundary $D$ such that $(X,D)$ is a Calabi-Yau pair and $\dim \mathcal{D}(X, D) = ~2$. To this aim, we will use the following lemma.
\begin{lem}[{\cite[Lemma 4.3]{ALP23}}] \label{lem:4.3}
 Let $D_1$ and $D_2$ be two normal surfaces with at worst du Val singularities in a smooth threefold X. Assume that scheme-theoretical intersection $D_1 \cap D_2$ is a rational curve $C$ of arithmetic genus 1 with one node $P \in C$. Also assume that $C$ belongs to the smooth locus of $D_1$ and $D_2$. Then $\dim \mathcal{D}(X, D) =2$, where $D = D_1 + D_2$.
\end{lem}
\subsection{Cusp singularities}
Recall that a surface singularity $p$ is called a cusp singularity, if $p \in \Sing X$, where $X$ is a normal surface and the preimage $\mu^{-1}(p)$ in the minimal resolution $\mu \colon \widetilde{X} \to X$ is either a cycle of smooth rational curves or a nodal rational curve of arithmetic genus one. It is known (\cite{Ka77}) that in $\mathbb{C}^{3}$ any cusp singularity up to a local analytic change of coordinates has the following form:
\begin{equation*} 
    T_{p,q,r} \colon\quad x^{q} + y^{p} + z^{r} + axyz = 0,\qquad 1/p + 1/q + 1/r < 1,\, a \neq 0.
\end{equation*}
%\begin{rem} \label{rem:2.1}
    %The equation $x^{q} + y^{p} + z^{r} + axyz = 0$, where $a \in \mathbb{C}^*$ takes the form $x'^{q} + y'^{p} + z'^{r} + x'y'z' = 0$ by multiplying each coordinate to a proper degree of $a$:
    %\begin{align*}
        %&x = a^{pr/(pqr - pq - qr - pr)}x',\\
        %&y = a^{qr/(pqr - pq - qr - pr)}y',\\
        %&z = a^{pq/(pqr - pq - qr - pr)}z'.
    %\end{align*}
%\end{rem} 
The minimal resolution graphs of the $T_{p,q,r}$ singularities are represented in Fig. \ref{fig:2.2} (see \cite{Lau77}). For the  $T_{2,3,7},\, T_{2,4,5},\, T_{3,3,4}$ cases, the exceptional divisor is a rational curve with one node.
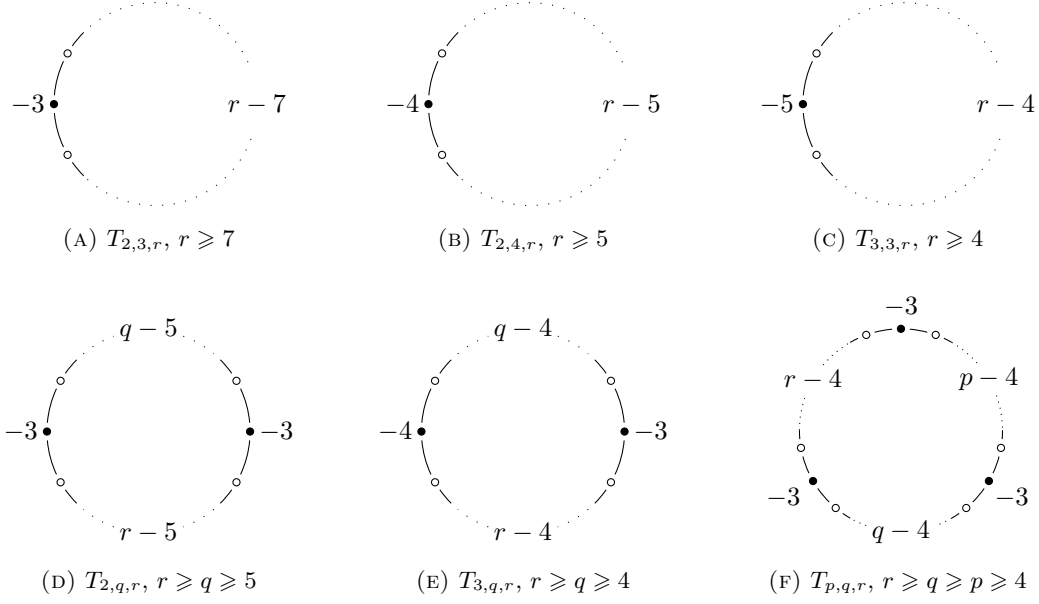
\begin{figure}
     \centering
     \begin{subfigure}[t]{0.3\textwidth}
         \centering
         \begin{tikzpicture}[scale = 0.45]
            \draw[loosely dotted] ({3*cos(130)},{3*sin(130)}) arc[start angle=130, end angle=20, radius=3];
            \draw[loosely dotted] ({3*cos(-20)},{3*sin(-20)}) arc[start angle=-20, end angle=-133, radius=3];
            \draw[black] ({3*cos(-135)},{3*sin(-135)}) arc[start angle=-135, end angle=-145, radius=3];
             \draw[black] ({3*cos(-155)},{3*sin(-155)}) arc[start angle=-155, end angle=-175, radius=3];
            \draw[black] ({3*cos(-185)},{3*sin(-185)}) arc[start angle=-185, end angle=-205, radius=3];
            \draw[black] ({3*cos(-215)},{3*sin(-215)}) arc[start angle=-215, end angle=-225, radius=3];
            \draw[black] ({3*cos(150)},{3*sin(150)}) circle (3pt);
            \draw[black] ({3*cos(-150)},{3*sin(-150)}) circle (3pt);
            \filldraw[black] ({3*cos(-180)},{3*sin(-180)}) circle (3pt);
            \node[anchor = east] at ({3*cos(-180)},{3*sin(-180)}) {$-3$};
            \node at ({3*cos(0)},{3*sin(0)}) {$r - 7$};
        \end{tikzpicture}
        \caption{$T_{2,3,r},\, r \geq 7$}
        \label{fig:t23r}
     \end{subfigure}
     \hfill
     \begin{subfigure}[t]{0.3\textwidth}
         \centering
         \begin{tikzpicture}[scale = 0.45]
            \draw[loosely dotted] ({3*cos(130)},{3*sin(130)}) arc[start angle=130, end angle=20, radius=3];
            \draw[loosely dotted] ({3*cos(-20)},{3*sin(-20)}) arc[start angle=-20, end angle=-133, radius=3];
            \draw[black] ({3*cos(-135)},{3*sin(-135)}) arc[start angle=-135, end angle=-145, radius=3];
             \draw[black] ({3*cos(-155)},{3*sin(-155)}) arc[start angle=-155, end angle=-175, radius=3];
            \draw[black] ({3*cos(-185)},{3*sin(-185)}) arc[start angle=-185, end angle=-205, radius=3];
            \draw[black] ({3*cos(-215)},{3*sin(-215)}) arc[start angle=-215, end angle=-225, radius=3];
            \draw[black] ({3*cos(150)},{3*sin(150)}) circle (3pt);
            \draw[black] ({3*cos(-150)},{3*sin(-150)}) circle (3pt);
            \filldraw[black] ({3*cos(-180)},{3*sin(-180)}) circle (3pt);
            \node[anchor = east] at ({3*cos(-180)},{3*sin(-180)}) {$-4$};
            \node at ({3*cos(0)},{3*sin(0)}) {$r - 5$};
        \end{tikzpicture}
        \caption{$T_{2,4,r},\, r \geq 5$}
        \label{fig:t24r}
     \end{subfigure}
     \hfill
     \begin{subfigure}[t]{0.3\textwidth}
         \centering
         \begin{tikzpicture}[scale = 0.45]
            \draw[loosely dotted] ({3*cos(130)},{3*sin(130)}) arc[start angle=130, end angle=20, radius=3];
            \draw[loosely dotted] ({3*cos(-20)},{3*sin(-20)}) arc[start angle=-20, end angle=-133, radius=3];
            \draw[black] ({3*cos(-135)},{3*sin(-135)}) arc[start angle=-135, end angle=-145, radius=3];
             \draw[black] ({3*cos(-155)},{3*sin(-155)}) arc[start angle=-155, end angle=-175, radius=3];
            \draw[black] ({3*cos(-185)},{3*sin(-185)}) arc[start angle=-185, end angle=-205, radius=3];
            \draw[black] ({3*cos(-215)},{3*sin(-215)}) arc[start angle=-215, end angle=-225, radius=3];
            \draw[black] ({3*cos(150)},{3*sin(150)}) circle (3pt);
            \draw[black] ({3*cos(-150)},{3*sin(-150)}) circle (3pt);
            \filldraw[black] ({3*cos(-180)},{3*sin(-180)}) circle (3pt);
            \node[anchor = east] at ({3*cos(-180)},{3*sin(-180)}) {$-5$};
            \node at ({3*cos(0)},{3*sin(0)}) {$r - 4$};
        \end{tikzpicture}
        \caption{$T_{3,3,r},\, r\geq 4$}
        \label{fig:t33r}
     \end{subfigure}
     \begin{subfigure}[b]{0.3\textwidth}
         \centering
         \begin{tikzpicture}[scale = 0.45]
            \draw[loosely dotted] ({3*cos(130)},{3*sin(130)}) arc[start angle=130, end angle=110, radius=3];
            \draw[loosely dotted] ({3*cos(70)},{3*sin(70)}) arc[start angle=70, end angle=50, radius=3];
            \draw[loosely dotted] ({3*cos(-50)},{3*sin(-50)}) arc[start angle=-50, end angle=-70, radius=3];
            \draw[loosely dotted] ({3*cos(-110)},{3*sin(-110)}) arc[start angle=-110, end angle=-130, radius=3];
            \draw[black] ({3*cos(-135)},{3*sin(-135)}) arc[start angle=-135, end angle=-145, radius=3];
             \draw[black] ({3*cos(-155)},{3*sin(-155)}) arc[start angle=-155, end angle=-175, radius=3];
            \draw[black] ({3*cos(-185)},{3*sin(-185)}) arc[start angle=-185, end angle=-205, radius=3];
            \draw[black] ({3*cos(-215)},{3*sin(-215)}) arc[start angle=-215, end angle=-225, radius=3];
            \draw[black] ({3*cos(150)},{3*sin(150)}) circle (3pt);
            \draw[black] ({3*cos(-150)},{3*sin(-150)}) circle (3pt);
            \filldraw[black] ({3*cos(-180)},{3*sin(-180)}) circle (3pt);
            \draw[black] ({3*cos(-45)},{3*sin(-45)}) arc[start angle=-45, end angle=-35, radius=3];
            \draw[black] ({3*cos(-25)},{3*sin(-25)}) arc[start angle=-25, end angle=-5, radius=3];
            \draw[black] ({3*cos(5)},{3*sin(5)}) arc[start angle=5, end angle=25, radius=3];
            \draw[black] ({3*cos(35)},{3*sin(35)}) arc[start angle=35, end angle=45, radius=3];
            \filldraw[black] ({3*cos(0)},{3*sin(0)}) circle (3pt);
            \draw[black] ({3*cos(30)},{3*sin(30)}) circle (3pt);
            \draw[black] ({3*cos(-30)},{3*sin(-30)}) circle (3pt);
            \node[anchor = east] at ({3*cos(-180)},{3*sin(-180)}) {$-3$};
            \node[anchor = west] at ({3*cos(0)},{3*sin(0)}) {$-3$};
            \node[white, anchor = south] at ({3*cos(90)},{3*sin(90) + 1}) {$-3$};
            \node at ({3*cos(90)},{3*sin(90)}) {$q-5$};
            \node at ({3*cos(-90)},{3*sin(-90)}) {$r - 5$};
        \end{tikzpicture}
        \caption{$T_{2,q,r},\,r\geq q\geq 5$}
        \label{fig:t2qr}
     \end{subfigure}
     \hfill
     \begin{subfigure}[b]{0.3\textwidth}
         \centering
         \begin{tikzpicture}[scale = 0.45]
            \draw[loosely dotted] ({3*cos(130)},{3*sin(130)}) arc[start angle=130, end angle=110, radius=3];
            \draw[loosely dotted] ({3*cos(70)},{3*sin(70)}) arc[start angle=70, end angle=50, radius=3];
            \draw[loosely dotted] ({3*cos(-50)},{3*sin(-50)}) arc[start angle=-50, end angle=-70, radius=3];
            \draw[loosely dotted] ({3*cos(-110)},{3*sin(-110)}) arc[start angle=-110, end angle=-130, radius=3];
            \draw[black] ({3*cos(-135)},{3*sin(-135)}) arc[start angle=-135, end angle=-145, radius=3];
             \draw[black] ({3*cos(-155)},{3*sin(-155)}) arc[start angle=-155, end angle=-175, radius=3];
            \draw[black] ({3*cos(-185)},{3*sin(-185)}) arc[start angle=-185, end angle=-205, radius=3];
            \draw[black] ({3*cos(-215)},{3*sin(-215)}) arc[start angle=-215, end angle=-225, radius=3];
            \draw[black] ({3*cos(150)},{3*sin(150)}) circle (3pt);
            \draw[black] ({3*cos(-150)},{3*sin(-150)}) circle (3pt);
            \filldraw[black] ({3*cos(-180)},{3*sin(-180)}) circle (3pt);
            \draw[black] ({3*cos(-45)},{3*sin(-45)}) arc[start angle=-45, end angle=-35, radius=3];
            \draw[black] ({3*cos(-25)},{3*sin(-25)}) arc[start angle=-25, end angle=-5, radius=3];
            \draw[black] ({3*cos(5)},{3*sin(5)}) arc[start angle=5, end angle=25, radius=3];
            \draw[black] ({3*cos(35)},{3*sin(35)}) arc[start angle=35, end angle=45, radius=3];
            \filldraw[black] ({3*cos(0)},{3*sin(0)}) circle (3pt);
            \draw[black] ({3*cos(30)},{3*sin(30)}) circle (3pt);
            \draw[black] ({3*cos(-30)},{3*sin(-30)}) circle (3pt);
            \node[anchor = east] at ({3*cos(-180)},{3*sin(-180)}) {$-4$};
            \node[anchor = west] at ({3*cos(0)},{3*sin(0)}) {$-3$};
            \node[white, anchor = south] at ({3*cos(90)},{3*sin(90) + 1}) {$-3$};
            \node at ({3*cos(90)},{3*sin(90)}) {$q-4$};
            \node at ({3*cos(-90)},{3*sin(-90)}) {$r - 4$};
        \end{tikzpicture}
        \caption{$T_{3,q,r},\, r \geq q \geq 4$}
        \label{fig:t3qr}
     \end{subfigure}
     \hfill
     \begin{subfigure}[b]{0.3\textwidth}
         \centering
         \begin{tikzpicture}[scale = 0.45]
            \draw[dotted] ({3*cos(60)},{3*sin(60)}) arc[start angle=60, end angle=40, radius=3];
            \draw[ dotted] ({3*cos(20)},{3*sin(20)}) arc[start angle=20, end angle=0, radius=3];
            \draw[ dotted] ({3*cos(-60)},{3*sin(-60)}) arc[start angle=-60, end angle=-70, radius=3];
            \draw[ dotted] ({3*cos(-110)},{3*sin(-110)}) arc[start angle=-110, end angle=-120, radius=3];
            \draw[ dotted] ({3*cos(180)},{3*sin(180)}) arc[start angle=180, end angle=160, radius=3];
            \draw[dotted] ({3*cos(140)},{3*sin(140)}) arc[start angle=140, end angle=120, radius=3];
            \draw[black] ({3*cos(85)},{3*sin(85)}) arc[start angle=85, end angle=75, radius=3];
             \draw[black] ({3*cos(65)},{3*sin(65)}) arc[start angle=65, end angle=60, radius=3];
            \draw[black] ({3*cos(0)},{3*sin(0)}) arc[start angle=0, end angle=-5, radius=3];
            \draw[black] ({3*cos(-15)},{3*sin(-15)}) arc[start angle=-15, end angle=-25, radius=3];
            \draw[black] ({3*cos(-35)},{3*sin(-35)}) arc[start angle=-35, end angle=-45, radius=3];
            \draw[black] ({3*cos(-55)},{3*sin(-55)}) arc[start angle=-55, end angle=-60, radius=3];
            \draw[black] ({3*cos(-120)},{3*sin(-120)}) arc[start angle=-120, end angle=-125, radius=3];
            \draw[black] ({3*cos(-135)},{3*sin(-135)}) arc[start angle=-135, end angle=-145, radius=3];
            \draw[black] ({3*cos(-155)},{3*sin(-155)}) arc[start angle=-155, end angle=-165, radius=3];
            \draw[black] ({3*cos(-175)},{3*sin(-175)}) arc[start angle=-175, end angle=-180, radius=3];
            \draw[black] ({3*cos(120)},{3*sin(120)}) arc[start angle=120, end angle=115, radius=3];
            \draw[black] ({3*cos(105)},{3*sin(105)}) arc[start angle=105, end angle=95, radius=3];
            \filldraw[black] ({3*cos(90)},{3*sin(90)}) circle (3pt);
            \filldraw[black] ({3*cos(-30)},{3*sin(-30)}) circle (3pt);
            \filldraw[black] ({3*cos(-150)},{3*sin(-150)}) circle (3pt);
            \draw[black] ({3*cos(70)},{3*sin(70)}) circle (3pt);
            \draw[black] ({3*cos(110)},{3*sin(110)}) circle (3pt);
            \draw[black] ({3*cos(-10)},{3*sin(-10)}) circle (3pt);
            \draw[black] ({3*cos(-50)},{3*sin(-50)}) circle (3pt);
            \draw[black] ({3*cos(-130)},{3*sin(-130)}) circle (3pt);
            \draw[black] ({3*cos(-170)},{3*sin(-170)}) circle (3pt);
            \node at ({3*cos(30)},{3*sin(30)}) {$p - 4$};
            \node at ({3*cos(150)},{3*sin(150)}) {$r - 4$};
            \node at ({3*cos(-90)},{3*sin(-90)}) {$q - 4$};
            \node[white, anchor = south] at ({3*cos(90)},{3*sin(90) + 1}) {$-3$};
            \node[anchor = south] at ({3*cos(90)},{3*sin(90) + 0.1}) {$-3$};
            \node[anchor = west] at ({3*cos(-30)-0.1},{3*sin(-30)-0.5}) {$-3$};
            \node[anchor = east] at ({3*cos(-150)-0.1},{3*sin(-150)-0.5}) {$-3$};
        \end{tikzpicture}
        \caption{$T_{p,q,r},\, r\geq q \geq p \geq 4$}
        \label{fig:tpqr}
     \end{subfigure}
        \caption{Minimal resolution graphs of the $T_{p,q,r}$ singularities. A vertex ~$\circ$ denotes a smooth rational $(-2)$-curve}. 
        \label{fig:2.2}
\end{figure}

We will prove the following lemma to recognize some particular singularities as $T_{p,q,r}$ singularities: 
\begin{lem} \label{lem:2.2}
    Suppose a surface in $\mathbb{C}^3$ is defined by the equation:
    \begin{equation} \label{eq:2.2}
    g(x,y,z) = xyz +x^p + y^q + z^r + g_{\geq 4}(x,y,z) = 0,
    \end{equation}
    where $1/p + 1/q + 1/r < 1,\, r \geq q \geq p \geq 4$ and $g_{\geq 4}$ is a polynomial whose monomials $c_{\alpha}x^{\alpha_1}y^{\alpha_2}z^{\alpha_3}$ have two properties: 
    \begin{enumerate}
        \item $\alpha_1 + \alpha_2 + \alpha_3 \geq 4$,
        \item at most one of the degrees $\alpha_1,\, \alpha_2,\, \alpha_3$ is equal to $0$.
    \end{enumerate} Then the singularity at the point $(0,0,0)$ is of the type $T_{p,q,r}$.
\end{lem}
\begin{proof}
   We describe the procedure which will give us the local analytic change of coordinates providing the $T_{p,q,r}$ form for the polynomial \eqref{eq:2.2}. Suppose the monomial $c_{\alpha}x^{\alpha_1}y^{\alpha_2}z^{\alpha_3}$ in $g$ has $c_{\alpha} \neq 0,\, \alpha_2 \geq 1,\, \alpha_3 \geq 1,\, \alpha_1 + \alpha_2 + \alpha_3 \geq 4$, then consider a local analytic change of coordinates $x = x' - c_{\alpha}x'^{\alpha_1}y^{\alpha_2-1}z^{\alpha_3-1}$. For \textit{any} monomial $c_{\beta}x^{\beta_1}y^{\beta_2}z^{\beta_3}$ in $g(x,y,z)$ we have
    \begin{gather*}
    c_{\beta}x^{\beta_1}y^{\beta_2}z^{\beta_3} = c_{\beta}(x' - c_{\alpha}x'^{\alpha_1}y^{\alpha_2-1}z^{\alpha_3-1})^{\beta_1}y^{\beta_2}z^{\beta_3} = \\ = c_{\beta}x'^{\beta_1}y^{\beta_2}z^{\beta_3} + f_{> \beta_1 + \beta_2 + \beta_3}(x', y, z),
   \end{gather*}
   where 
   \begin{equation*}
   f_{> \beta_1 + 
   \beta_2 + \beta_3} =
    \begin{cases}
    0, &\beta_1 = 0; \\
    \sum \limits_{k = 1}^{\beta_1} c_{\alpha}^kc_{\beta} \binom{\beta_1}{k}x^{\beta_1 + k(\alpha_1-1)}y^{\beta_2 +k(\alpha_2 - 1)}z^{\beta_3 + k(\alpha_3 - 1)}, &\beta_1 \geq 1.
    \end{cases}   
   \end{equation*}
   In particular, 
    \begin{equation*}
        xyz = x'yz - c_\alpha x'^{\alpha_1}y^{\alpha_2}z^{\alpha_3}.
    \end{equation*}
   Note that if $f_{>\beta_1 + \beta_2 + \beta_3} \neq 0$, the degrees of the monomials in $f_{>\beta_1 + \beta_2 + \beta_3}$ are strictly greater than $\beta_1 + \beta_2 + \beta_3$:
   \begin{gather*}
        \beta_1 + k(\alpha_1-1) +  \beta_2 +k(\alpha_2 - 1) +  \beta_3 + k(\alpha_3 - 1)) = \\ = \beta_1 + \beta_2 + \beta_3 + k(\alpha_1 + \alpha_2 + \alpha_3) - 3k \geq \\ \geq \beta_1 + \beta_2 + \beta_3 + k > \beta_1 + \beta_2 + \beta_3.
   \end{gather*}
   If $\beta_1 + \beta_2 + \beta_3 \geq 4$, they are greater than $\alpha_1 + \alpha_2 + \alpha_3$ as well:
   \begin{gather*}
        \beta_1 + k(\alpha_1-1) +  \beta_2 +k(\alpha_2 - 1) +  \beta_3 + k(\alpha_3 - 1)) = \\ = \beta_1 + \beta_2 + \beta_3 + (k-1)(\alpha_1 + \alpha_2 + \alpha_3) + \alpha_1 + \alpha_2 + \alpha_3 - 3k \geq \\ \geq 4 + 4(k-1) - 3k + \alpha_1 + \alpha_2 + \alpha_3 = \\ = k + \alpha_1 + \alpha_2 + \alpha_3 > \alpha_1 + \alpha_2 + \alpha_3.   
   \end{gather*}
   Also note that monomials in $f_{>\beta_1 + \beta_2 + \beta_3}$ satisfy condition (2) of this lemma. Indeed, if $\beta_i + k(\alpha_i - 1) = 0$ for $2$ different indices $i \in \{1,2,3\}$, then for the same indices $i$ we have $\alpha_i = 0$, which is impossible. From here we conclude that the change of coordinates eliminated the monomial $c_{\alpha}x^{\alpha_1}y^{\alpha_2}z^{\alpha_3}$, did not change other monomials and possibly added monomials of degree $> \alpha_1 + \alpha_2 + \alpha_3$ which satisfy conditions $(1)$--$(2)$ of this lemma. We had $\alpha_2 \geq 1,\, \alpha_3 \geq 1$, but two other cases are treated analogously. Use this procedure inductively to sequentially eliminate monomials of all degrees. Take the composition of changes of coordinates:
\begin{align*}
x &= x'' + h_1(x'',y'',z''), \\
y &= y'' + h_2(x'',y'',z''), \\
z &= z'' + h_3(x'',y'',z''), \\
x''y''z'' &+ x''^4 + y''^4 + z''^4 = 0.
\end{align*} Here $h_i$ are the power series with zero constant and linear terms, so we have an isomorphism $\mathbb{C}[[x,y,z]] \cong \mathbb{C}[[x'', y'', z'']]$ which shows that the point is a singularity of the type $T_{p,q,r}$.
\end{proof}
\section{Sextic double solid}
In this section, we construct a smooth Fano threefold $X$ in the family $1.1$ and a boundary $D$ on it such that $\mathrm{coreg}(X, D)=0$. 
Consider a surface $S \subseteq \P{3}$ of degree $6$ which is defined by the following equation:
\begin{equation*}   f_6 = x_0^2x_1^2x_2^2+x_0^2x_1^4+x_2^ 6+x_3(x_0^5+x_1^5+x_3^5) = 0.
\end{equation*}
Consider a threefold $X$ given by the equation 
\[
\{u^2 = f_6\}\subseteq \mathbb{P}(1,1,1,1,3),
\]
where $u$ has weight $3$ and $x_i$ have weight $1$. 
Note that its anti-canonical linear system $\left|-K_X\right|$ defines a double cover $\phi\colon X \to \P{3}$ ramified in $S$.
\begin{prop}
    $X$ is a smooth Fano threefold from the family 1.1. 
\end{prop} 
\begin{proof}
By adjunction formula we have that $X$ is a Fano variety. 
Let us check that $X$ is smooth.
It is sufficient to check that $S$ is smooth surface. We use Gr\"obner basis computation to make sure that the Jacobian matrix has rank $1$ at every point of $X$. This is equivalent to showing that the following system of equations have no solutions in \P{3}:
\begin{gather}
\label{sys20}
\begin{cases}
    \frac{\partial f_6}{\partial x_0} = 2x_0x_1^2x_2^2 +2x_0x_1^4+5x_3x_0^4 = 0,\\
    \frac{\partial f_6}{\partial x_1} = 2x_1x_0^2x_2^2 + 4x_1^3x_0^2 + 5x_3x_1^4 = 0,\\ 
    \frac{\partial f_6}{\partial x_2} =2x_2x_0^2x_1^2 + 6x_2^5 = 0,\\
    \frac{\partial f_6}{\partial x_3} = 6x_3^5 + x_0^5+ x_1^5 = 0.
\end{cases}
\end{gather}
Consider an ideal $\mathfrak{a} =  (\frac{\partial f_6}{\partial x_i})_{i=0}^3 \subseteq \mathbb{C}[x_0,x_1,x_2,x_3]$. The Gr\"obner basis computation Comp.~\ref{fig:calc2} shows that $\mathfrak{a} + (x_3 - 1) = (1)$. Therefore, for the solution of the system \eqref{sys20} it is only possible to have $x_3 = 0$. A Gr\"obner basis for the ideal $\mathfrak{a} + (x_3)$ is computed in Comp.~\ref{fig:calc1}. In particular, 
\begin{equation*}
    \{x_3,\, x_2^6,\, 2x_1^8+x_1^6x_2^2,\, x_0^5+x_1^5\} \subseteq \mathfrak{a} + (x_3).    
\end{equation*} From here we conclude that the only solution of the system \eqref{sys20} is $x_3 = x_2 = x_1 = x_0 = 0$, or there is no solutions in $\P{3}$.
\end{proof}
Now we construct the boundary divisor $D$ on $X$. 
Consider an element $D \in |-K_X|$ such that $\phi(D) = H = \{x_3 = 0\}$. Then $\phi$ induces a double cover $\phi|_{D}\colon D \to H$ ramified in a sextic curve 
\[
R = H \cap S = \{x_0^2x_1^2x_2^2+x_0^2x_1^4+x_2^ 6 = 0\}.
\]
Note that the curve $R$ has two singular points $p = (1:0:0:0),\, q = (0:1:0:0)$. In the chart $U_0 = \{x_0 \neq 0\}$ the equation of $R$ near the point $p$ is of the following form: 
\begin{equation*}
    x_1^2x_2^2+x_1^4+x_2^6 = 0.
\end{equation*} Then $D$ is locally defined near the point $p' = \phi^{-1}(p)$ as $u^2 = x_1^2x_2^2 + x_1^4 + x_2^6$. Make a change of coordinates $u = iu',\, x_1 = ix_1'$ to get the following form of the equation of $D$: 
\begin{equation*}
    u'^2 - x_1'^2x_2^2 + x_1^4 + x_2^6 = 0.
\end{equation*}
Then make a local analytic change of coordinates $u' = u'' + x_1'x_2$ to obtain the equation
\begin{equation} \label{eq:3.3}
    u''^2 + x_1'^4 + x_2^6 + 2u''x_1'x_2 = 0.
\end{equation}
Observe that $p'\in D$ is a
$T_{2,4,6}$ singularity. 
Next, in the chart $U_1 = \{x_1 \neq 0\}$ the equation of $H$ near the point $q$ is of the following form: 
\begin{equation*}
    x_0^2x_2^2+x_0^2+x_2^6 = 0.
\end{equation*}
Locally near the point $q' = \phi^{-1}(q)$ the divisor $D$ is defined by the equation 
\[
u^2 =  x_0^2x_2^2+x_0^2+x_2^6.
\]
Make an invertible linear change of coordinates $x = u - x_0,\,y =  u + x_0,\, z = x_2$ to obtain:
\begin{equation*}
    xy = \frac{1}{2}xz^2 - \frac{1}{2}yz^2 + z^6.
\end{equation*}
Now make the following local analytic change of coordinates:
\begin{align*}
    &x' = x + \frac{1}{2} z^2,\\
    &y' = y - \frac{1}{2} z^2,\\
    &z' = z\left(-\frac{1}{4} + z^2\right)^{1/4}.
\end{align*}
This provides the form $x'y' = z'^4$, so $q'\in D$ is locally-analytically equivalent to the $A_3$ singularity. Therefore, the divisor $D$ has two singular points: $p'$ of type $T_{2,4,6}$ and $q'$ of type $A_3$.
\begin{prop}
The pair $(X,D)$ is a Calabi-Yau pair with $\dim \mathcal{D}(X, D) =~2$.
\end{prop}
\begin{proof}
The pair $(X,D)$ is lc by inversion of adjunction.  Since $D \in \left|-K_X\right|$, the pair is Calabi-Yau. Now consider the weighted blow-up of the point $p'$ with weights $(2,1,1)$ with respect to the local parameters $(u'',x_1', x_2)$ as in \eqref{eq:3.3}:
\begin{equation*}
    \sigma_1 \colon X' \to X.
\end{equation*}
Let $D'$ be the strict transform of $D$ and $E_1 \cong \mathbb{P}(2,1,1)$ the exceptional divisor. Note that $K_{X'} = \sigma_1^*K_X + 3E_1,\, D' = \sigma_1^*D - 4E_1$, so the pair $(X', D' + E_1)$ is the log pullback of $(X,D)$. The exceptional divisor $E_1$ has one $A_1$ singularity which does not lie on $D' \cap E_1$. $D'$ has one $A_1$ singularity at the point $p''$ and $D' \cap E_1 = C_1$, where $C_1$ is a nodal curve which has one node at $p''$. Blow up $p''$ to obtain a morphism 
\[\sigma_2\colon X''\to X'.\] Let $D''$, $E_1'$ be the strict transforms of $D'$ and $E_1$, respectively, and $E_2 \cong \P{2}$ the exceptional divisor. The pair $(X'', D'' + E_1' + E_2)$ is the log pullback of the pair ($X, D$). The map $\sigma_2|_{D''}\colon D'' \to D$ provides the minimal resolution of the $T_{2,4,6}$ singularity at the point $p$ (see Fig. \ref{fig:t24r}). The pair $(X'', D'' + E_1' + E_2)$ is dlt and the boundary has a zero-dimensional stratum, so $\dim \mathcal{D}(X, D) =~2$.
\end{proof}

\section{Quartic in \P{4}}
Examples of smooth quartic hypersurfaces in \P{4} with maximal intersection boundaries were provided in \cite{Ka18}. There also exists a quartic with 3 ordinary double points which has $\coreg(X) = 0$, since its hyperplane section has $T_{3,3,4}$ singularity \cite[3.2]{Ka18}. We provide an example of a smooth quartic $X_4 \subseteq \P{4}$ with  singularity of the type $T_{3,3,4}$ on the hyperplane section, which is defined by the following equation:
\begin{equation} \label{eq:4.1}
    f = x_0x_1^3 + x_0x_2^3 + x_3^4 + x_0x_1x_2x_3 + x_4(x_0^3 + 2x_1^3 + x_4^3) = 0.
\end{equation}
\begin{prop}
    $X_4$ is a smooth Fano threefold from the family 1.2.
\end{prop} 
\begin{proof}
By adjunction formula, $X_4$ is a Fano variety. We need to check whether the Jacobian matrix has rank 1 at every point of $X_4$, which is equivalent to showing that the following system of equations has no solutions in \P{4}:
\begin{gather*}
\label{sys30}
\begin{cases}
    \pfrac{f}{x_0} = x_1^3 + x_2^3 + x_1x_2x_3 + 3x_0^2x_4 = 0,\\
    \pfrac{f}{x_1} = 3x_1^2x_0 + x_0x_2x_3 + 2x_4x_1^2 = 0,\\ 
    \pfrac{f}{x_2} =3x_2^2x_0 + x_0x_1x_3 = 0,\\
    \pfrac{f}{x_3} = 4x_3^3+x_0x_1x_2 = 0,\\
    \pfrac{f}{x_4} =x_0^3 + 2x_1^3 + 4x_4^3 = 0.
\end{cases}
\end{gather*}
Consider an ideal $\mathfrak{a} =  (\frac{\partial f}{\partial x_i})_{i=0}^4 \subseteq \mathbb{C}[x_0,x_1,x_2,x_3, x_4]$. The Gr\"obner basis computation Comp.~\ref{fig:calc4} shows that $\mathfrak{a} + (x_4 - 1) = (1)$. Therefore, for the solution of the system \eqref{sys20} it is only possible to have $x_4 = 0$. A Gr\"obner basis for the ideal $\mathfrak{a} + (x_4)$ is computed in Comp.~\ref{fig:calc3}. In particular, 
\begin{equation*}
    \{x_4,\, x_3^4,\, x_2^7,\, x_1^3 + x_2^3 + x_1x_2x_3,\, x_0^3 - 2x_2^3 - 2 x_1x_2x_3\} \subseteq \mathfrak{a} + (x_4).    
\end{equation*} From here we conclude that the only solution of the system \eqref{sys20} is $x_4 = x_3 = x_2 = x_1 = x_0 = 0$, so there is no solutions in $\P{4}$. 
\end{proof}
Now we consider a hyperplane section $H = X_4 \cap \{x_4 = 0\}$ as a boundary divisor.
\begin{prop} \label{prop:3_3}
    The pair $(X_{4},H)$ is a Calabi-Yau pair with $\dim \mathcal{D}(X_{4}, H) =~2.$ 
\end{prop}
\begin{proof}
    The equation which defines the hyperplane section $H$ in $\P{3} = \{x_4 = 0\}$ is obtained by substituting $x_4 = 0$ into $\eqref{eq:4.1}$
    \begin{equation*}
        g = x_0x_1^3 + x_0x_2^3 + x_3^4 + x_0x_1x_2x_3 = 0.    
    \end{equation*}
    Find all singular points of this surface from the following system of equations:
    \begin{gather*}
    \label{sys31}
    \begin{cases}
        \pfrac{g}{x_0} = x_1^3 + x_2^3 + x_1x_2x_3 = 0,\\
        \pfrac{g}{x_1} = 3x_1^2x_0 + x_0x_2x_3  = 0,\\ 
        \pfrac{g}{x_2} = 3x_2^2x_0 + x_0x_1x_3 = 0,\\
        \pfrac{g}{x_3} = 4x_3^3+x_0x_1x_2 = 0.
    \end{cases}
    \end{gather*}
    Solving the system, we get 4 points: $p = (1:0:0:0:0),\, q_k =  (0:1:-\varepsilon_3^{k}:0:0)$, where $\varepsilon_3$ is a primitive 3th root of unity and $k = 0,\,1,\, 2$. In the chart $U_0 = \{x_0 \neq 1\}$ we see that point $p$ is of the type $T_{3,3,4}$:
    \begin{equation*}
         x_1^3 + x_2^3 + x_3^4 + x_1x_2x_3 = 0.   
    \end{equation*}
   Points $q_k$ are singularities of the type $A_3$, which can be seen in the chart $U_1 = \{x_1 \neq 0\}$, where the hyperplane section is defined as follows:
   \begin{equation*}
    x_0 + x_0x_2^3 + x_3^4 + x_0x_2x_3 = 0.     
   \end{equation*} Make the change of coordinates $x_2 = x_2' - \varepsilon_3^k$ so the point $q_k$ has coordinates $(0,0,0)$:
    \begin{equation*}
        x_0(3\varepsilon_3^{2k}x_2' - \varepsilon_3^k x_3) - 3x_0x_2'^2\varepsilon_3^k + x_0x_2'x_3 +  x_0x_2'^3 + x_3^4 = 0.
    \end{equation*}
    Now make the following invertible linear change of coordinates: 
    \begin{align*}
    &x = x_0, \\
    &y = 3\varepsilon_3^{2k}x_2' - \varepsilon_3^k x_3, \\
    &z = 3\varepsilon_3^{2k}x_2' + \varepsilon_3^k x_3,
    \end{align*}
    \begin{equation*}
        xy - \frac{xy^2}{6} - \frac{xyz}{6} + x\sfrac{y+z}{6}^3 + \varepsilon_3^{2k}\sfrac{z-y}{2}^4 = 0.
    \end{equation*}
    This singularity is locally-analytically equivalent to the $x'y' = z'^4$ singularity, which can be shown similarly to the proof of the Lemma \ref{lem:2.2}. Therefore, $H$ has one $T_{3,3,4}$ and three $A_3$ singularities.
    
    Thus, $(X_4,\, H)$ is an lc pair by inversion of adjunction. Since $H \in \left|-K_{X_4}\right|$, the pair is Calabi-Yau. Now consider a blow up of the point $p$:
    \begin{equation*}
        \sigma \colon X' \to X_4.
    \end{equation*}
    X is a smooth threefold and exceptional divisor is $E \cong \P{2}$. This blow up provides the minimal resolution of the $T_{3,3,4}$ singularity. Note that the proper transform of $H$ is a smooth divisor $D$ and intersection $D \cap E$ is a rational curve of arithmetic genus one with one node. Thus, by Lemma \ref{lem:4.3} we have
    $\dim \mathcal{D}(X', D + E) =  \dim \mathcal{D}(X_4, H) =  2$.
\end{proof}
\section{Intersection of a quadric and a cubic in \P{5}}
In this section, we construct a smooth Fano threefold $X_{2,3}$ in the family $1.3$ and a boundary $H$ on it such that $\mathrm{coreg}(X_{2,3}, H)=0$. Let \[X_{2,3} = \{f_1 = f_2 = 0\}\subseteq \P{5}\] be the intersection of the following quadric and cubic:
\begin{gather}
\label{sys0}
\begin{cases}
    x_0^2 + x_1x_2 = x_3^2+x_4^2+x_5^2,\\
    x_0^3 + x_0x_1^2+x_0x_2^2+x_1x_2^2+x_3x_4x_5 = 0.
\end{cases}
\end{gather}
\begin{prop} $X_{2,3}$ is a smooth Fano threefold from the family 1.3. 
\end{prop}
\begin{proof}
By adjunction formula, $X_{2,3}$ is a Fano variety. To prove that $X_{2,3}$ is smooth, it is sufficient to show that for all $x \in X_{2,3}$ with the homogeneous coordinates $x_i$, the Jacobian matrix has rank 2:
\begin{gather*}
    \begin{bmatrix}
        Df_1 \\
        Df_2 \\
    \end{bmatrix} = 
    \begin{bmatrix}
        & 2x_0 & x_2 & x_1 & -2x_3 & -2x_4 & -2x_5 &\\
        & 3x_0^2 + x_1^2 +x_2^2 & 2x_1x_0+x_2^2 & 2x_2(x_0+x_1) & x_4x_5 & x_3x_5 & x_3x_4 &\\
    \end{bmatrix}.
\end{gather*}
Note that the quadric $\{f_1 = 0\}$ is smooth, in particular $Df_1(x) \neq 0$ for any $x \in X_{2,3}$. The qubic $\{f_2 = 0\}$ has three singular points:
\begin{gather*}
p_1 = (0:0:0:1:0:0),\\
p_2 = (0:0:0:0:1:0),\\
p_3 = (0:0:0:0:0:1).
\end{gather*}
They don't lie on the quadric, so $Df_2(x) \neq 0$ for any $x \in X_{2,3}$.
Therefore, if the Jacobian matrix is degenerate on $X_{2,3}$, then $Df_1 = \lambda Df_2,\, \lambda \neq 0,\, \lambda \in \mathbb{C}$. Write this condition as the following system of equations:
\begin{equation} \label{sys1}
\begin{cases}{}
      x_0^2 + x_1x_2 = x_3^2+x_4^2+x_5^2,\\
      x_0^3 + x_0x_1^2+x_0x_2^2+x_1x_2^2+x_3x_4x_5 = 0,\\
      2x_0 = \lambda (3x_0^2 + x_1^2 +x_2^2), \\
      x_2 = \lambda (2x_1x_0+x_2^2), \\
      x_1 = 2\lambda x_2(x_0+x_1), \\
      -2x_3 = \lambda x_4x_5, \\ 
      -2x_4 = \lambda x_3x_5, \\ 
      -2x_5 = \lambda x_3x_4.     
\end{cases}
\end{equation}
We can simplify this system. Make the following substitution:
\begin{equation*}
\begin{aligned}
    & \alpha_1 = x_3 + x_4 + x_5, \\
    & \alpha_2 = x_3x_4 + x_3x_5 + x_4x_5, \\
    & \alpha_3 = x_3x_4x_5, \\
    & x_3^2 + x_4^2 + x_5^2 = \alpha_1^2 - 2\alpha_2.\\
\end{aligned}
\end{equation*}
Then from the last three equations of \eqref{sys1} we have: 
\begin{equation}
\begin{aligned} \label{sys:5.1}
     &-2\alpha_1 = \lambda \alpha_2, \\
     &-8\alpha_3 = \lambda^3 \alpha_3^2, \\
     &4\alpha_2 = \lambda^2 \alpha_3\alpha_1,\\
     &3\alpha_3\lambda = -2(\alpha_1^2 -2\alpha_2).
\end{aligned}
\end{equation}
From here, since $\alpha_3(\lambda^3\alpha_3 + 8) = 0$, we need to consider two cases:\\
1)\, $\alpha_3 = 0.$ \\
In this case from \eqref{sys:5.1} we get $\alpha_2 = \alpha_1 = 0$ and we can rewrite the system \eqref{sys1} as:
\begin{equation*}
\begin{cases} 
      x_0^2 + x_1x_2 = 0,\\
      x_0^3 + x_0x_1^2+x_0x_2^2+x_1x_2^2 = 0, \\
      2x_0 = \lambda (3x_0^2 + x_1^2 +x_2^2), \\
      x_2 = \lambda (2x_1x_0+x_2^2), \\
      x_1 = 2\lambda x_2(x_0+x_1). \\   
\end{cases}
\end{equation*}
Consider the ideal $\mathfrak{a} \subseteq \mathbb{C}[x_0, x_1, x_2, \lambda]$ generated by the polynomials of the system. A Gr\"obner basis of the ideal $\mathfrak{a}$ is $\{x_0, x_1, x_2\}$, as computed in Comp.~\ref{fig:calc5}. then $x_0 = x_1 = x_2 = 0$. Since $\alpha_3 = \alpha_1 = \alpha_2 = 0$, we get $x_3 = x_4 = x_5 = 0$, so we have no solutions in $\P{6}$.
\\
2)\, $\alpha_3 = -\frac{8}{\lambda^3}$. \\
In this case we have the system:
\begin{equation*}
\begin{cases} 
      \lambda(x_0^2 + x_1x_2) = \lambda\alpha_1^2 +4\alpha_1,\\
      \lambda^3 (x_0^3 + x_0x_1^2+x_0x_2^2+x_1x_2^2) - 8  = 0, \\
      2x_0 = \lambda (3x_0^2 + x_1^2 +x_2^2), \\
      x_2 = \lambda (2x_1x_0+x_2^2), \\
      x_1 = 2\lambda x_2(x_0+x_1), \\
      2\alpha_1^2\lambda^2+8\lambda\alpha_1-24 = 0.
\end{cases}
\end{equation*}
Let $\mathfrak{b} \subseteq \mathbb{C}[x_0,x_1,x_2,\alpha_1, \lambda]$ be the ideal generated by the polynomials of the system. By the computation Comp.~\ref{fig:calc6}, $\mathfrak{b} = (1)$, so the system is incompatible. Therefore, $X_{2,3}$ is a smooth threefold. 
\end{proof}
Now we consider the hyperplane section $H = X_{2,3} \cap \{x_0 = 0\}$ as a boundary divisor.
\begin{prop} \label{prop:5.5}
The pair $(X_{2,3},H)$ is a Calabi-Yau pair with $\dim \mathcal{D}(X_{2,3}, H) =2$. 
\end{prop}
\begin{proof}
Firstly, we show that the hyperplane section has two singular points $p = (0:1:0:0:0:0)$ and $q = (0:0:1:0:0:0)$. We check where the following Jacobian matrix has rank 2:
\begin{gather*}
    \begin{bmatrix}
         Df_1 \\
         Df_2 \\
         Dx_0 \\ 
    \end{bmatrix} = 
    \begin{bmatrix}
        & 2x_0 & x_2 & x_1 & -2x_3 & -2x_4 & -2x_5 &\\
        & 3x_0^2 + x_1^2 +x_2^2 & 2x_1x_0+x_2^2 & 2x_2(x_0+x_1) & x_4x_5 & x_3x_5 & x_3x_4 &\\
        & 1 & 0 & 0 & 0 & 0 & 0 &\\
    \end{bmatrix}.
\end{gather*}
Since the first two rows are linearly independent on $X_{2,3}$, we have $\lambda Df_1 + \mu Df_2 = Dx_0$ for some $\lambda,\, \mu \in \mathbb{C}$. Write this condition as the following system of equations:
\begin{equation} 
\begin{cases}{}
     x_1x_2 = x_3^2+x_4^2+x_5^2,\\
     x_1x_2^2+x_3x_4x_5 = 0,\\
        \mu(x_1^2 + x_2^2) = 1, \\
      \lambda x_2  + \mu x_2^2 = 0, \\
      \lambda x_1 + 2\mu x_2x_1 = 0, \\
      -2 \lambda x_3 + \mu x_4x_5 = 0, \\ 
      -2 \lambda x_4 + \mu x_3x_5 = 0, \\ 
      -2 \lambda x_5 + \mu x_3x_4 = 0.  
\end{cases}
\end{equation}
Consider the ideal $\mathfrak{a} \subseteq \mathbb{C}[x_1, x_2, x_3, x_4, x_5, \lambda, \mu]$ generated by the polynomials of the system. A Gr\"obner basis for this ideal is computed in Comp.~\ref{fig:calc7}. In particular, 
\begin{equation*}
    \{x_5^4,\, x_4^4,\, x_3^4,\, x_1x_2 - x_3^2 - x_4^2 - x_5^2\} \subseteq \mathfrak{a}.
\end{equation*}
So, $x_0 = x_3 = x_4 = x_5 = 0$ and $x_1x_2 = 0$. The solutions are as follows: 
\begin{align*}
    &p = (0:1:0:0:0:0),\quad \lambda = 0,\,\,\,\,\, \mu = 1, \\ &q = (0:0:1:0:0:0),\quad \lambda = -1,\, \mu = 1.  
\end{align*}
Note that in the chart $U_1 = \{x_1 \neq 0\}$ we can substitute $x_2 = x_3^2 + x_4^2 + x_5^2$ into the qubic equation \eqref{sys0} and get the equation of $X_{2,3} \cap H$ near the point $p = (0:1:0:0:0:0)$:
\begin{equation*}
    x_3x_4x_5 + x_3^4 + x_4^4 + x_5^4 + 2(x_3^2x_4^2 + x_3^2x_5^2 + x_4^2x_5^2) = 0.
\end{equation*}
By Lemma \ref{lem:2.2} the point $p$ is a singularity of the type $T_{4,4,4}$. 

Now we can see that $q$ is an ordinary double point by substituting $x_1 = x_3^2 + x_4^2 + x_5^2$ in the chart $U_2 = \{x_2 \neq 0\}$ into the cubic equation \eqref{sys0}:
\begin{equation*}
    x_3^2 + x_4^2 + x_5^2 + x_3x_4x_5 = 0.
\end{equation*} 
Since the corank of the second differential of this singularity is zero, we recognize it as an ordinary double point. Hence the pair $(X_{2,3},\,H)$ is lc by inversion of adjunction. Since $H \in \left|-K_{X_{2,3}}\right|$, the pair is Calabi-Yau. Consider a blow up of the point $p$:
\[\sigma_1 \colon X' \to X_{2,3}.\]
Let $H'$ be a strict transform of the $H$ and let $E \cong \P{2}$ be an exceptional divisor. Then the pair $(X',E+H')$ is the log pullback of the pair $(X_{2,3}, H)$. Note that the restriction $\sigma_1|_{H'} \colon H' \to H$ provides the minimal resolution of the $T_{4,4,4}$ singularity from Fig. \ref{fig:tpqr}. The intersection $H'\cap E$ is a cycle of three smooth rational curves $C_1 + C_2 + C_3$. Let $p_1, p_2, p_3$ be the intersection points of the each pair of curves $C_i\cap C_j$. Consider a composition of blow ups in the points $p_i,\, i = 1,\,2,\,3$:
\[\sigma_2 \colon X'' \to X'.\]
Let $E_i \cong \P{2}$ be the exceptional divisors over $p_i,\, i = 1,\, 2,\, 3$, and let $H'',\, E'$ be the strict transforms of the divisors $H',\, E$, respectively. The pair $(X'',H'' + E')$ is the log pullback of the pair $(X', H' + E)$. Note that $E_i$ intersect $H''$ and $E'$ in a common curve: $L_i = E' \cap E_i = H'' \cap E_i$. Consider a composition of blow ups of the curves $L_i,\, i = 1,\,2,\,3$:
\[\sigma_3 \colon X''' \to X''.\]
Let $F_i$ be the exceptional divisors over $L_i,\, i = 1,\, 2,\, 3$, and $H''',\, E''$ -- the strict transforms of the divisors $H'',\, E'$, respectively. The pair $(X''',H''' + E'' 
+ F_1 + F_2 + F_3)$  is the log pullback of the pair $(X'', H'' + E')$. Note that the pair $(X''',H''' + E'' 
+ F_1 + F_2 + F_3)$ is dlt and its boundary admits a zero-dimensional stratum, so the equality $\dim \mathcal{D}(X_{2,3}, H) = 2$ follows. 
\end{proof}
\section{Intersection of three quadrics in \P{6}}
In this section, we construct a smooth Fano threefold $X_{2,2,2}$ in the family $1.4$ and a boundary $H$ on it such that $\mathrm{coreg}(X_{2,2,2}, H)=0$. Let \[X_{2,2,2}  =  \{f_1 = f_2 = f_3 = 0\}\subseteq \P{6}\] be the intersection of the three following quadrics :
\begin{gather}
\label{sys50}
\begin{cases}
    x_0x_2 + x_1x_3 = x_4^2+x_5^2+x_6^2,\\
    x_0x_3 + x_1x_2 = x_5x_6,\\
    x_3^2 + x_2x_4 + x_0(2x_0 + x_1 + x_5 + x_6) = 0.
\end{cases}
\end{gather}
\begin{prop} $X_{2,2,2}$ is a smooth Fano threefold from the family 1.4. 
\end{prop}
\begin{proof}
By adjunction formula, $X_{2,2,2}$ is a smooth Fano variety. To prove that $X_{2,2,2}$ is smooth it is sufficient to show that for all $x \in X_{2,2,2}$ with homogeneous coordinates $x_i$, the Jacobian matrix has rank 3:
\begin{gather*}
    \begin{bmatrix}
         Df_1 \\
         Df_2  \\
         Df_3  \\
    \end{bmatrix} = 
    \begin{bmatrix}
        & x_2 & x_3 & x_0 & x_1 & -2x_4 & -2x_5 & -2x_6 &\\
        & x_3 & x_2 & x_1 & x_0 & \,0 & -x_6 &-x_5&\\
        & 4x_0 + x_1 + x_5 + x_6 & x_0 & x_4 & 2x_3 & \,\,x_2 & \phantom{-}x_0 & \phantom{-}x_0 &\\
    \end{bmatrix}.
\end{gather*}
Notice that the quadric $\{f_1 = 0\}$ is smooth, so $Df_1 \neq 0$ on $X_{2,2,2}$. Then notice that the quadric $\{f_2 = 0\}$ is of corank $1$ and has one singular point $(0:0:0:1:0:0:0)$, but this point does not lie on $X_{2,2,2}$, so $Df_2 \neq 0$. We check that $Df_1$ and $Df_2$ are linearly independent on $X_{2,2,2}$. Suppose $Df_1 = \lambda Df_2$, $\lambda \neq 0$, then
\begin{equation} \label{eq:6.3}
\begin{cases}
    x_3 = \lambda x_2, \\
    x_2 = \lambda x_3, \\
    x_1 = \lambda x_0, \\
    x_0 = \lambda x_1, \\
    0= - 2 \lambda x_4, \\
    -2x_5 = -\lambda x_6, \\
    -2x_6 = -\lambda x_5,
\end{cases} \Rightarrow \quad
\begin{cases}
    x_3(1-\lambda^2) = 0,\\
    x_1(1-\lambda^2) = 0,\\
    x_5(4-\lambda^2) = 0, \\
    x_4 = 0.
\end{cases}
\end{equation}
Both $x_1$ and $x_3$ cannot be zero because then $x_2 = x_0 = x_4 = 0$ from \eqref{eq:6.3} and $x_5 = x_6 = 0$ from the $f_2$ equation \eqref{sys50}. So, $1-\lambda^2 = 0$, then $4 - \lambda^2 \neq 0$, and then $x_5 = x_6 = 0$. Now we substitute $x_4 = x_5 = x_6 = 0$ to $f_1=0$ and $f_2 = 0$ to get 
\begin{equation*}
    x_0x_3(1 + \lambda^2) = 0,\qquad x_0x_2(1 + \lambda^2) = 0.
\end{equation*}
Note that $1 + \lambda^2\neq0$ and  $x_0 \neq 0$ because then from $f_3=0$ we get $x_3^2 = 0$, so $x_3 = x_2 = x_1 = x_0 = x_4 = x_5 = x_6 = 0$. So, $x_2 = x_3 = 0$ and since $\lambda^2 = 1$, we get two candidates: 
\begin{align*}
    &p_{1,2} = (1:\pm1:0:0:0:0:0),
\end{align*}
but they don't satisfy equation $f_3 = 0$, so $Df_1 \neq \lambda Df_2$ on $X_{2,2,2}$. Now we show that $Df_3 \notin ~\left\langle Df_1,\,Df_2\right\rangle$. Suppose $Df_3 = \lambda Df_1 + \mu Df_2$ for some $\lambda,\,\mu \in \mathbb{C}$. This relation gives the system of equations:
\begin{equation*}
\begin{cases}
    x_0x_2 + x_1x_3 = x_4^2+x_5^2+x_6^2,\\
    x_0x_3 + x_1x_2 = x_5x_6,\\
    x_3^2 + x_2x_4 + x_0(2x_0 + x_1 + x_5 + x_6) = 0,\\
    4x_0 + x_1 + x_5 + x_6 = \lambda x_2 + \mu x_3, \\
    x_0 = \lambda x_3 + \mu x_2,\\
    x_4 = \lambda x_0 + \mu x_1, \\
    2x_3 = \lambda x_1 + \mu x_0, \\
    x_2 = -2\lambda x_4, \\
    x_0 = -2\lambda x_5 - \mu x_6, \\
    x_0 = -2\lambda x_6 - \mu x_5.
    \end{cases}
\end{equation*}
Consider the ideal $\mathfrak{a} \subseteq \mathbb{C}[x_0, x_1, x_2, x_3, x_4, x_5, x_6, \lambda, \mu]$ generated by the polynomials of the system. A Gr\"obner basis for this ideal is computed in Comp.~\ref{fig:calc8}. In particular, 
\begin{equation*}
    \{x_0^2,\, x_1^2,\, x_2^2,\, x_3^2,\, x_4^2,\, x_5x_6,\, x_5^2 + x_6^2\} \subseteq \mathfrak{a}.
\end{equation*}
This shows that $x_i = 0,\, \forall\, i = 0,\ldots,6$ so there is no solutions in $\P{6}$ and $X_{2,2,2}$ is smooth.
\end{proof}
Now we consider the hyperplane section $H = X_{2,2,2} \cap \{x_0 = 0\}$ as a boundary divisor. In $\P{5} = \{x_0 = 0\}$ the equations of $H$ are as follows:
\begin{equation*}
    \begin{cases}
        x_1x_3 = x_4^2+x_5^2+x_6^2,\\
        x_1x_2 = x_5x_6,\\
        x_3^2 + x_2x_4 = 0.    
    \end{cases}
\end{equation*}
\begin{prop}
The pair $(X_{2,2,2},H)$ is a Calabi-Yau pair with $\dim \mathcal{D}(X_{2,2,2}, H) =~2.$ 
\end{prop}
\begin{proof}
Firstly, we show that the hyperplane section has two singular points:
\[
    p = (0:1:0:0:0:0:0),\quad q = (0:0:1:0:0:0:0).
\] We check where the following Jacobian matrix has rank 3:
\begin{gather*}
    \begin{bmatrix}
         Df_1 \\
         Df_2  \\
         Df_3  \\
         Dx_0
    \end{bmatrix} = 
    \begin{bmatrix}
        & x_2 & x_3 & x_0 & x_1 & -2x_4 & -2x_5 & -2x_6 &\\
        & x_3 & x_2 & x_1 & x_0 & \,0 & -x_6 &-x_5&\\
        & 4x_0 + x_1 + x_5 + x_6 & x_0 & x_4 & 2x_3 & \,\,x_2 & \phantom{-}x_0 & \phantom{-}x_0 &\\
        & 1 & 0 & 0 & 0 & 0 & \phantom{-}0 & \phantom{-}0 &\\
    \end{bmatrix}
\end{gather*}
Since $Df_i$ are linearly independent, we have $Dx_0 = \lambda Df_1 + \mu Df_2 + \nu Df_3$ for some $\lambda,\, \mu,\, \nu \in ~\mathbb{C}$. This relation gives us the following system of equations:
\begin{equation*}
\begin{cases}
    x_1x_3 = x_4^2+x_5^2+x_6^2,\\
    x_1x_2 = x_5x_6,\\
    x_3^2 + x_2x_4= 0,\\
    \lambda x_2 + \mu x_3 +  \nu (x_1 + x_5 + x_6) = 1, \\
    \lambda x_3 + \mu x_2 = 0,\\
    \mu x_1 + \nu x_4 = 0, \\
    \lambda x_1 + 2\nu x_3 = 0, \\
    -2\lambda x_4 + \nu x_2 = 0, \\
    -2\lambda x_5 -\mu x_6 = 0, \\
    -2\lambda x_6 -\mu x_5 = 0.
    \end{cases}
\end{equation*}
Consider the ideal $\mathfrak{a} \subseteq \mathbb{C}[x_1, x_2, x_3, x_4, x_5, x_6, \lambda, \mu, \nu]$ generated by the polynomials of the system. A Gr\"obner basis for this ideal is computed in Comp.~\ref{fig:calc9}. In particular, 
\begin{equation*}
    \{x_3^3,\, x_4^4,\, x_5^4 ,\,x_6^4,\, x_1x_2 - x_5x_6\} \subseteq \mathfrak{a}.
\end{equation*}
This shows that $x_i = 0\,\,$ for $i = 3,\,\ldots,\,6$ and $x_1x_2 = 0$, and we get two singular points $p,\,q$:
\begin{align*}
    &p = (0:1:0:0:0:0:0), \quad \lambda = 0,\, \mu = 0,\, \nu = 1, \\
    &q = (0:0:1:0:0:0:0), \quad \lambda = 1,\, \mu = 0,\, \nu = 0.
\end{align*}
In the chart $U_1 = \{x_1 \neq 0\}$ we can substitute $x_3 = x_4^2 + x_5^2 + x_6^2$ and $x_2 = x_5x_6$ into the $f_3$ equation \eqref{sys50} to get the equation of $H$ near the point $p$:
\begin{equation*}
    x_4x_5x_6 + x_4^4 + x_5^4 + x_6^4 + 2(x_4^2x_5^2 + x_4^2x_6^2 + x_5^2x_6^2) = 0.
\end{equation*}
By Lemma \ref{lem:2.2}  $p$ is a singularity of the type $T_{4,4,4}$. In the chart $U_2 = \{x_2 \neq 0\}$ we can substitute $x_1 = x_5x_6$ and $x_4 = -x_3^2$ into $f_1$ to get the equation of $H$ near point $q$: 
\begin{equation*}
    x_5^2 + x_6^2 - x_3x_5x_6 + x_3^4 = 0.
\end{equation*}
Similarly to the proof of Lemma \ref{lem:2.2}, we can conclude that $q$ is an $A_3$ singularity. Hence the pair $(X_{2,2,2},\,H)$ is lc by inversion of adjunction. Since $H \in \left|-K_{X_{2,2,2}}\right|$, the pair is Calabi-Yau. Now  take the composition of blow ups as in the proof of the Proposition \ref{prop:5.5} to observe that $\dim \mathcal{D}(X_{2,2,2}, H) = 2$.
\end{proof}
\newpage
\begin{appendix}
\section{Gr\"obner basis computations}

\renewcommand{\figurename}{Computation}

All the computations were performed in Wolfram Mathematica.
\begin{figure}[h]
\begin{center}
\begin{verbatim}
    GroebnerBasis[
    {
    x3, 
    2*x0*x1^2*x2^2 + 2*x0*x1^4+5*x3*x0^4,
    2*x1*x0^2*x2^2 + 4*x1^3*x0^2 + 5*x3*x1^4,
    2*x2*x0^2*x1^2+6*x2^5,
    6*x3^5 + x0^5+x1^5
    },
    {x0,x1,x2,x3}
    ]
\end{verbatim}
\begin{align*}
    \mathfrak{a} + (x_3)= (&x_3,\, x_2^6,\, x_1^2x_2^5,\, x_1^6x_2^3,\, x_1^7x_2^2,\, 2x_1^8+x_1^6x_2^2,\, x_0x_1^4 + x_0x_1^2x_2^2, \\ &x_0^2x_1x_2^3-6x_1x_2^5,\, x_0^2x_1^2x_2+3x_2^5,\, 2x_0^2x_1^3+x_0^2x_1x_2^2, \\&-x_1^7x_2 + 3x_0^3x_2^5,\, x_0^5+x_1^5)
\end{align*}
\caption{}
\label{fig:calc1}
\end{center}
\end{figure}
\begin{figure}[h]
\begin{center}
\begin{verbatim}
    GroebnerBasis[
    {
    x3 - 1, 
    2*x0*x1^2*x2^2 + 2*x0*x1^4+5*x3*x0^4,
    2*x1*x0^2*x2^2 + 4*x1^3*x0^2 + 5*x3*x1^4,
    2*x2*x0^2*x1^2+6*x2^5,
    6*x3^5 + x0^5+x1^5
    },
    {x0,x1,x2,x3}
    ]
\end{verbatim}
\begin{align*}
    \mathfrak{a} + (x_3 - 1) = (1)
\end{align*}
\caption{}
\label{fig:calc2}
\end{center}
\end{figure}
\begin{figure}[h]
\begin{center}
\begin{verbatim}
    GroebnerBasis[
    {
    x4, 
    x1^3 + x2^3 + x1*x2*x3 + 3*x0^2*x4,
    3*x1^2*x0 + x0*x2*x3 + 6*x4*x1^2,
    3*x0*x2^2 + x0*x1*x3,
    4*x3^3 + x0*x1*x2,
    x0^3 + 2*x1^3 + 4*x4^3
    },
    {x0,x1,x2,x3}
    ]
\end{verbatim}
\begin{align*}
    \mathfrak{a} + (x_4)= (&x_4,\, x_3^4,\, x_2^2x_3^3,\, x_2^4x_3^2,\, x_2^5x_3,\, x_2^7,\, x_1x_2x_3^3,\, x_1x_2^3x_3^2,\, x_2^6 + x_1x_2^4x_3,\, x_1x_2^5,\, \\ & x_1^2x_3^3, 3x_2^5 + 4x_1x_2^3x_3 + x_1^2x_2x_3^2,\, 3x_1^2x_2^3 - 2x_2^4x_3-2x_1x_2^2x_3^2,\, \\ &x_1^3 + x_2^3 + x_1x_2x_3,\, x_0x_2x_3^2 + 36x_2x_3^3,\, x_0x_2^2x_3-12x_1x_3^3,\, x_0x_2^3,\, \\ & 3x_0x_2^2 + x_0x_1x_3,\, x_0x_1x_2 + 4x_3^3,\, 3x_0x_1^2 + x_0x_2x_3,\, \\ &x_1x_2^4 + x_1^2x_2^2x_3 + 2x_0^2x_3^3,\, x_0^3 - 2x_2^3 - 2 x_1x_2x_3)
\end{align*}
\caption{}
\label{fig:calc3}
\end{center}
\end{figure}
\begin{figure}[h]
\begin{center}
\begin{verbatim}
    GroebnerBasis[
    {
    x4 - 1, 
    x1^3 + x2^3 + x1*x2*x3 + 3*x0^2*x4,
    3*x1^2*x0 + x0*x2*x3 + 6*x4*x1^2,
    3*x0*x2^2 + x0*x1*x3,
    4*x3^3 + x0*x1*x2,
    x0^3 + 2*x1^3 + 4*x4^3
    },
    {x0,x1,x2,x3}
    ]
\end{verbatim}
\begin{equation*}
    \mathfrak{a} + (x_4 - 1) = (1)
\end{equation*}
\caption{}
\label{fig:calc4}
\end{center}
\end{figure}
\begin{figure}[h]
\begin{center}
\begin{verbatim}
    GroebnerBasis[
    {  
    x0^2 + x1*x2,
    x0^3 + x0*x1^2 + x1*x2^2 + x0*x2^2,
    2x0 - \[Lambda](3x0^2 + x1^2 + x2^2),
    x2 - \[Lambda]( 2x0*x1 + x2^2),
    x1 - 2\[Lambda]*x2*(x0 + x1)
    },
    {x0,x1,x2,\[Lambda]}
    ]
\end{verbatim}
\begin{equation*}
    \mathfrak{a} = (x_2,\, x_1,\, x_0)
\end{equation*}
\caption{}
\label{fig:calc5}
\end{center}
\end{figure}
\begin{figure}[h]
\begin{center}
\begin{verbatim}
    GroebnerBasis[
    {
    \[Lambda](x0^2 + x1*x2) - \[Lambda]*\[Alpha]1^2 - 4\[Alpha]1,
    \[Lambda]^3*(x0^3 + x0*x1^2 + x1*x2^2 + x0*x2^2) - 8,
    2x0 - \[Lambda]*(3x0^2 + x1^2 + x2^2),
    x2 - \[Lambda]*(2x0*x1 + x2^2),
    x1 - 2\[Lambda]*x2*(x0 + x1),
    2\[Alpha]1^2\[Lambda]^2 + 8\[Lambda]\[Alpha]1 - 24 
    },
    {x0,x1,x2,\[Lambda],\[Alpha]1}
    ]
\end{verbatim}
\begin{equation*}
    \mathfrak{b} = (1)
\end{equation*}
\caption{}
\label{fig:calc6}
\end{center}
\end{figure}
\begin{figure}[h]
\begin{center}
\begin{verbatim}
    GroebnerBasis[
    {
    x1*x2 - x3^2 - x4^2 - x5^2,
    x1*x2^2 + x3*x4*x5,
    \[Mu]*(x1^2 + x2^2) - 1,
    \[Lambda]*x2 + \[Mu]*(x2^2),
    \[Lambda]*x1 + 2*\[Mu]*x2*x1,
    -2*\[Lambda]*x3  + \[Mu]*x4*x5,
    -2*\[Lambda]*x4  + \[Mu]*x3*x5,
    -2*\[Lambda]*x5  + \[Mu]*x3*x4
    },
    {x1,x2,x3,x4,x5,\[Lambda],\[Mu]}
    ]
\end{verbatim}
\begin{align*}
    \mathfrak{a} = (&\lambda^4 - \lambda^2\mu,\, x_5\lambda^2,\, x_5^2\lambda,\, x_5^4,\, x_4\lambda^2,\, x_4x_5\lambda,\, x_4x_5^2,\, x_4^2\lambda,\, x_4^2x_5,\, x_4^4,\, \\ & 2x_3\lambda - x_4x_5\mu, -2x_4\lambda + x_3x_5\mu,\, x_3x_5^2,\, -2x_5\lambda + x_3x_4\mu, x_3x_4x_5,\, \\ &x_3x_4^2,\, x_3^2x_5,\, x_3^2x_4,\, x_3^4,\, \lambda^3 + \lambda\mu + 2x_2\mu^2,\, \lambda^2 + x_2\lambda\mu,\, -\lambda+x_2\lambda^2- \\ &- 2x_2\mu,\, x_3x_4 + 4x_2x_5,\, 4x_2x_4+x_3x_5,\,4x_2x_3+x_4x_5,\, x_2\lambda + x_2^2\mu,\, \\&x_2^2 + x_2^3\lambda,\, x_1\lambda + 2x_3^2\mu + 2x_4^2\mu + 2x_5^2\mu,\, x_3x_4 + 4x_1x_5^3\mu,\, 4x_3^3+x_1x_4x_5,\,\\& x_3x_5+4x_1x_4^3\mu,\, 4x_4^3+x_1x_3x_5,\, x_1x_3x_4 + 4x_5^3,\, -x_2-x_2^2\lambda + x_1x_3^2\mu + \\ &+x_1x_4^2\mu + x_1x_5^2\mu,\, x_1x_2-x_3^2-x_4^2-x_5^2,\, -1-x_2\lambda+x_1^2\mu)
\end{align*}
\caption{}
\label{fig:calc7}
\end{center}
\end{figure}
\begin{figure}[h]
\begin{center}
\begin{verbatim}
    GroebnerBasis[
    {
    x0*x2 + x1*x3 - x4^2 -x5^2 - x6^2,
    x0*x3 + x1*x2 - x5*x6,
    x3^2 + x2*x4 + x0*(2x0 + x1 + x5 + x6),
    4x0 + x1 + x5 + x6 - \[Lambda]*x2 -\[Mu]*x3,
    x0- \[Lambda]*x3 - \[Mu]*x2,
    x4  - \[Lambda]*x0 - \[Mu]*x1,
    2x3  - \[Lambda]*x1 - \[Mu]*x0,
    x2 +  2\[Lambda]*x4,
    x0 + 2\[Lambda]*x5 + \[Mu]*x6,
    x0 + 2\[Lambda]*x6 + \[Mu]*x5
    },
    {x0,x1,x2,x3,x4,x5,x6,\[Lambda],\[Mu]},
    MonomialOrder -> DegreeReverseLexicographic
    ]
\end{verbatim}
\begin{align*}
    \mathfrak{a} = (&x_0 + 2x_6\lambda + x_5\mu,\, x_0 + 2x_5\lambda + x_6\mu,\, x_2 + 2x_4\lambda,\, -x_0 + x_3\lambda  + x_2\mu,\, \\&-4x_0-x_1-x_5-x_6 + x_2\lambda  + x_3\mu,\, -2x_3 + x_1 \lambda  + x_0\mu,\, -x_4 + \\& + x_0\lambda  + x_1\mu,\, x_5x_6,\, x_5^2 + x_6^2,\, x_4x_5 + x_4x_6,\, x_3x_5 + x_3x_6,\, x_2x_5 + x_2x_6,\, \\&x_1x_5 + x_1x_6,\, x_0x_5 + x_0x_6,\, x_4^2,\, x_3x_4,\, x_2x_4,\, x_1x_4,\, x_0x_4,\, x_3^2,\, x_2x_3,\, x_1x_3,\, \\&x_0x_3,\, x_2^2,\, x_1x_2,\, x_0x_2,\, x_1^2,\, x_0x_1,\, x_0^2,\, x_0x_6 + x_6^2\mu,\,x_2x_6 + x_4x_6\mu,\, \\&-14x_0x_6 - 4x_1x_6 - 4x_6^2 + 3x_3x_6\mu,\, 4x_0x_6 + 2x_1x_6 + 2 x_6^2 + 3x_2x_6\mu,\, \\&4x_3x_6 - 4x_4x_6 + 3x_1x_6\mu,\, -8x_3x_6 + 2x_4x_6 + 3x_0x_6\mu,\, x_6^3,\, x_4x_6^2,\, x_3x_6^2,\, \\&x_2x_6^2,\, x_1x_6^2,\, x_0x_6^2)
\end{align*}
\caption{}
\label{fig:calc8}
\end{center}
\end{figure}
\begin{figure}[h]
\begin{center}
\begin{verbatim}
    GroebnerBasis[
    {
    x1*x3 - x4^2 -x5^2 - x6^2,
    x1*x2 - x5*x6,
    x3^2 + x2*x4 ,
    \[Nu]*(x1 + x5 + x6) + \[Lambda]*x2 +\[Mu]*x3 - 1,
    \[Lambda]*x3 + \[Mu]*x2,
    \[Nu]*x4  + \[Mu]*x1,
    \[Nu]*2*x3  + \[Lambda]*x1 ,
    \[Nu]*x2 -  2*\[Lambda]*x4,
    -2*\[Lambda]*x5 - \[Mu]*x6,
    -2*\[Lambda]*x6 - \[Mu]*x5
    },
    {x1,x2,x3,x4,x5,x6,\[Lambda],\[Mu],\[Nu]}
    ]
\end{verbatim}
\begin{align*}
    \mathfrak{a} = (&\mu^4,\, 2\mu^3 + \lambda\mu\nu,\, \lambda\mu^3,\, 2\lambda\mu^2 + \lambda^2\nu,\, x_6\mu^2,\, x_6\lambda\mu,\, x_6\lambda^2,\, x_6^2\mu,\, x_6^2\lambda,\, \\&x_6\lambda + 2x_6^3\nu^3,\, x_6^4,\, -2x_6\lambda + x_5\mu,\, 2x_5\lambda - x_6\mu,\, -4\mu^3 + x_5x_6\nu^5,\, \\&x_5x_6^2,\, 2\mu^2 + \lambda\nu - 2x_6\lambda\nu^2 - x_6\mu\nu^2 + 2x_5^2\nu^4 + 2x_6^2\nu^4,\, x_5^2x_6,\, \\&x_6\mu + 4x_5^3\nu^3,\, x_5^4,\, \mu\nu-2x_6\lambda\nu^2 - x_6\mu\nu^2 + x_4\
    \nu^3,\, 4x_4\mu^2 + x_5x_6\nu^3,\, \\&2x_4\lambda\nu - x_5x_6\nu^3,\, x_4\lambda\mu,\, x_4\lambda^3 + \mu^2 + x_4\mu\nu^2,\, x_4x_6 - 4x_5^3\mu,\,\\& x_4x_5 - 4x_6^3\nu,\, x_4\mu + x_4^2\nu^2,\, 4x_4^2\mu - x_5x_6\nu,\, x_4^2\lambda,\, x_5x_6 + 4x_4^3\nu,\, x_4^4,\, \\&2x_6\lambda + 4x_4\mu + x_6\mu + 4x_3\nu - 4x_5^2\nu^2 - 4x_6^2\nu^2,\, -4x_4\lambda + 4x_3\mu + \\& + x_5x_6\nu^2,\, x_3\lambda^2 + \mu - 2x_6\lambda\nu - x_6\mu\nu + x_4\nu^2,\, x_3x_6 - x_6^3\nu,\, \\& x_3x_5 - x_5^3\nu,\, 4x_3x_4 + x_5x_6,\, x_3^3,\, -2x_4\lambda + x_2\nu,\, x_3\lambda + x_2\mu,\, \\&-\lambda + x_2\lambda^2 + x_4\lambda^2 + 2 x_6\lambda\nu + 2 x_4\mu\nu + x_6\mu\nu - 2 x_5^2\nu^3 - 2x_6^2\nu^3,\, \\& x_2x_6,\, x_2x_5,\, x_3^2 + x_2x_4,\, -x_3 + x_2x_3\lambda + x_4^2\nu + x_5^2\nu + x_6^2\nu + x_5^3\nu^2 + \\& + x_6^3\nu^2,\, -x_2 + x_2^2\lambda - x_3^2\lambda + x_5x_6\nu,\, -4 + 4x_2\lambda + 4 x_4\lambda + 4x_1\nu + \\& + 4x_5\nu + 4x_6\nu - x_5x_6\nu^2,\,  
    x_1\mu + x_4\nu,\, 2x_1\lambda - 2x_6\lambda - 4x_4\mu - \\& - x_6\mu + 4x_5^2\nu^2 + 4x_6^2\nu^2,\, 4x_4^3 + x_1x_5x_6,\, x_1x_3 - x_4^2 - x_5^2 - x_6^2,\,\\& x_1x_2 - x_5x_6)
\end{align*}
\caption{}
\label{fig:calc9}
\end{center}
\end{figure}
\end{appendix}

\end{document}